\documentclass[a4paper,11pt]{amsart}
\usepackage[english]{babel}
\usepackage{amssymb}
\usepackage{amscd}
\usepackage{mathtools}
\usepackage[paper=a4paper,left=3cm,right=3cm,top=2cm,bottom=2.0cm]{geometry}
\usepackage{enumerate}
\usepackage{hyperref}
\usepackage{tikz}
\usepackage{tikz-cd}

\newtheorem{thm}{Theorem}[section]
\newtheorem*{thm-non}{Theorem}

\newtheorem{lem}[thm]{Lemma}
\newtheorem{prop}[thm]{Proposition}
\newtheorem{cor}[thm]{Corollary}
\theoremstyle{definition}
\newtheorem{defi}[thm]{Definition}
\newtheorem{rem}[thm]{Remark}

\DeclareMathOperator\Br{Br}
\DeclareMathOperator\End{End}
\DeclareMathOperator\Hom{Hom}
\DeclareMathOperator\Ext{Ext}
\DeclareMathOperator\Gal{Gal}
\DeclareMathOperator\ext{ext}

\DeclareMathOperator\id{id}

\DeclareMathOperator\M{M}

\DeclareMathOperator\Fix{Fix}

\DeclareMathOperator\Ima{Im}

\DeclareMathOperator\Pic{Pic}
\DeclareMathOperator\rk{rk}
\DeclareMathOperator\ord{ord}

\DeclareMathOperator\Aut{Aut}

\begin{document}

\title{Picard schemes of noncommutative bielliptic surfaces}

\author{Fabian Reede}
\address{Institut f\"ur Algebraische Geometrie, Leibniz Universit\"at Hannover, Welfengarten 1, 30167 Hannover, Germany}
\email{reede@math.uni-hannover.de}

\begin{abstract}
We study the nontrivial elements in the Brauer group of a bielliptic surface and show that they can be realized as Azumaya algebras with a simple structure at the generic point of the surface. We go on to study some properties of the noncommutative Picard scheme associated to such an Azumaya algebra.
\end{abstract}


\maketitle

\vspace{-7.ex}
\section*{Introduction}
According to Enriques' classification of smooth complex algebraic surfaces, the surfaces with Kodaira dimension zero can be divided into four classes: K3 surfaces, Enriques surfaces, abelian surfaces and bielliptic surfaces, see \cite{ven}. 

The study of moduli spaces of sheaves on surfaces with Kodaira dimension zero gave rise to a lot of interesting results, for example the construction of hyperk\"{a}hler varieties, that is irreducible holomorphic symplectic manifolds, of higher dimension. But it seems that the case of bielliptic surfaces was not studied extensively in this direction.

This situation changed recently. On the one hand, Nuer studied stable sheaves and especially possible Chern characters of stable sheaves on bielliptic surface in detail, see \cite{nuer}. On the other hand, building on Beauville's work in the case of Enriques surfaces in \cite{beau}, Bergstr{\"o}m, Ferrari, Tirabassi and Vodrup studied the so-called Brauer map for bielliptic surfaces in \cite{tirab}.

In this article we want to combine both directions by studying a certain version of noncommutative Picard schemes. Here we think of a noncommutative variety as a pair $(X,\mathcal{A})$ consisting of a classical complex algebraic variety $X$ and a sheaf of noncommutative $\mathcal{O}_X$-algebras $\mathcal{A}$ of finite rank as an $\mathcal{O}_X$-module. The algebras of interest in this article are Azumaya algebras. These are locally isomorphic to a matrix algebra $M_r(\mathcal{O}_X)$ with respect to the \'{e}tale topology and they are classified by the Brauer group $\Br(X)$ of $X$.

A noncommutative Picard scheme $\Pic(\mathcal{A})$ is the moduli scheme $\M_{\mathcal{A}/X}$ of certain sheaves on $X$, which have the structure of a left $\mathcal{A}$-module. These moduli schemes were constructed by Hoffmann and Stuhler in \cite{hoff}.

In this article we study the situation of noncommutative bielliptic surfaces. The main results of this article can be summarized as follows:

\begin{thm-non}
	Let $X$ be a bielliptic surface such that the Brauer group is nontrivial. Then every nontrivial element in $\Br(X)$ can be represented by an Azumaya algebra $\mathcal{A}$ that is a generically central simple cyclic division algebra. 
	
	Let $(X,\mathcal{A})$ be a noncommutative bielliptic surface defined by such an algebra. If the Brauer map of $X$ is injective then we have:
	\begin{enumerate}[i)]
		\item The noncommutative Picard scheme $\Pic(\mathcal{A})$ is smooth.
		\item Every torsion free $\mathcal{A}$-module of rank one can be deformed into a locally projective $\mathcal{A}$-module, that is the locus $\Pic(\mathcal{A})^{lp}$ of locally projective $\mathcal{A}$-modules is dense in $\Pic(\mathcal{A})$.
	\end{enumerate}
	Let $\overline{X}$ be the canonically covering abelian surface and denote the pullback of the Azumaya algebra to $\overline{X}$ by $\overline{\mathcal{A}}$, then $\Pic(\overline{\mathcal{A}})$ has a symplectic structure. For fixed Chern classes $c_1$ and $c_2$ we have
	\begin{enumerate}[i)]
		\setcounter{enumi}{2}
		\item $\Pic(\mathcal{A})_{c_1,c_2}$ is a finite \'{e}tale cover of a smooth projective subscheme $Y$ in $\Pic(\overline{\mathcal{A}})_{\overline{c_1},\overline{c_2}}$.
		\item The subscheme $Y$ is Lagrangian if and only if the canonical cover of $X$ has degree two or $\dim(\Pic(\mathcal{A})_{c_1,c_2})=1$.
	\end{enumerate} 
\end{thm-non}

Most results in this article are direct counterparts or have immediate generalizations from the case of noncommutative Enriques surfaces studied in \cite{fr}. But some results are new due to new phenomena on bielliptic surfaces, for example Brauer classes of order three. In this article we work over the field of complex numbers $\mathbb{C}$.

\section{Modules over an Azumaya algebra and cyclic Galois coverings}\label{1}
In this section we generalize the results of \cite[Section 1]{fr} from \'{e}tale double covers to arbitrary cyclic \'{e}tale Galois covers. So denote by $W$ a smooth projective variety of dimension $d$ together with a nontrivial $n$-torsion line bundle $L$, that is $n$ is the order of $L$ in $\Pic(W)$. By \cite[I.17]{hulek} there is a cyclic \'{e}tale Galois cover
\begin{equation}\label{push}
	q: \overline{W} \rightarrow W\,\,\,\text{such that}\,\,\,q_{*}\mathcal{O}_{\overline{W}}\cong\bigoplus\limits_{i=0}^{n-1} L^i.
\end{equation}
For every coherent sheaf $E$ on $W$ we denote by $\overline{E}$ the pullback of $E$ to $\overline{W}$ along $q$. 

\begin{defi}
	A sheaf of $\mathcal{O}_W$-algebras $\mathcal{A}$ is called an Azumaya algebra if
	\begin{itemize}
		\item $\mathcal{A}$ is locally free of finite rank and
		\item for every point $w\in W$ the fiber $\mathcal{A}(w)$ is a central simple algebra over the residue field $\mathbb{C}(w)$.
	\end{itemize}
	Furthermore a coherent $\mathcal{O}_W$-module $E$ is said to be an Azumaya module or an $\mathcal{A}$-module if $E$ has the structure of a left $\mathcal{A}$-module. 
\end{defi}

Azumaya algebras on $W$ are classified up to similarity by the Brauer group $\Br(W)$. Here similarity for two Azumaya algebras $\mathcal{A}$ and $\mathcal{B}$ is defined as follows:
\begin{equation*}
	\mathcal{A}\sim \mathcal{B}\,\,\,\text{if}\,\,\, \mathcal{A}\otimes\mathcal{E}nd_W(\mathcal{E})\cong\mathcal{B}\otimes\mathcal{E}nd_W(\mathcal{F}),
\end{equation*}
where $\mathcal{E}$ and $\mathcal{F}$ are locally free $\mathcal{O}_W$-modules of finite rank. 

We say $\mathcal{A}$ is trivial if $\left[\mathcal{A} \right]=\left[\mathcal{O}_W\right]$ in $\Br(W)$. A quick computation shows that $\mathcal{A}$ is trivial if and only if $\mathcal{A}\cong \mathcal{E}nd_W(P)$ for some locally free sheaf $P$ of finite rank. From now on, if not otherwise stated, an Azumaya algebra $\mathcal{A}$ is always a nontrivial Azumaya algebra such that $\overline{\mathcal{A}}$ is also nontrivial. The rank of an Azumaya algebra $\mathcal{A}$ is always a square so it makes sense to define the degree of such an algebra by:
\begin{equation*}
	\deg(\mathcal{A}):=\sqrt{\rk(\mathcal{A})}.
\end{equation*} 

Using \eqref{push}, the proof of \cite[Lemma 1.4]{fr} immediately gives: 

\begin{lem}\label{hom}
	Assume $E$ and $F$ are $\mathcal{A}$-modules, then 
	\begin{equation*}
		\Hom_{\overline{\mathcal{A}}}(\overline{E},\overline{F})\cong \bigoplus\limits_{i=0}^{n-1}\Hom_{\mathcal{A}}(E,F\otimes L^i).
	\end{equation*}
\end{lem}

Applying Lemma \ref{hom} to the case $F=E$ we find:

\begin{cor}\label{homvan}
	Assume $E$ is an $\mathcal{A}$-module. $\overline{E}$ is a simple $\overline{\mathcal{A}}$-module if and only if $E$ is a simple $\mathcal{A}$-module and $\Hom_{\mathcal{A}}(E,E\otimes L^i)=0$ for $1\leqslant i\leqslant n-1$.
\end{cor}

We also have the following variant of Serre duality, see \cite[Proposition 3.5.]{hoff}:
\begin{prop}\label{serre}
	Assume $E$ and $F$ are $\mathcal{A}$-modules, then for $i\geqslant 0$ there are isomorphisms 
	\begin{equation*}
		\Ext^i_{\mathcal{A}}(E,F)\cong \left( \Ext^{d-i}_{\mathcal{A}}(F,E\otimes\omega_{W})\right)^{\vee}.
	\end{equation*}
\end{prop}

In the case of surfaces, that is $\dim(W)=2$, we find similar to \cite[Lemma 1.7]{fr}:

\begin{lem}\label{double2}
	Assume $E$ is an $\mathcal{A}$-module which is torsion free as an $\mathcal{O}_W$-module. If $\overline{E^{**}}$ is a simple $\overline{\mathcal{A}}$-module, then $\Hom_{\mathcal{A}}(E,E^{**})\cong\mathbb{C}$ and for $1\leqslant i \leqslant n-1$:
	\begin{equation*}
		\Hom_{\mathcal{A}}(E,E^{**}\otimes L^i)=0.
	\end{equation*}
\end{lem}

Recall that the relative automorphism group of the \'{e}tale cyclic Galois cover $q: \overline{W} \rightarrow W$ is generated by a covering map $\iota$ of order $n$:
\begin{equation*}
	\Aut(\overline{W}/W)=\left\langle \iota \right\rangle \cong\mathbb{Z}/n\mathbb{Z}.
\end{equation*}

As the group $\Aut(\overline{W}/W)$ is cyclic, the descent condition for a coherent sheaf $F$ on $\overline{W}$, see   \cite[\href{https://stacks.math.columbia.edu/tag/0D1V}{Lemma 0D1V}]{stacks-project}, reduces to the existence of an isomorphism $\varphi_{\iota}: F\rightarrow \iota^{*}F$ such that the map:
\begin{equation*}\label{desccond}
	\psi:=\left( \iota^{n-1}\right)^{*}\varphi_{\iota}\circ\ldots\circ \iota^{*}\varphi_{\iota}\circ\varphi_{\iota}: F \rightarrow \left(\iota^n \right)^{*}F \cong F
\end{equation*}
is the identity map.
If $F$ is simple, then any $\varphi_{\iota}$ satisfies $\psi\in \End_{\overline{W}}(F)=\mathbb{C}\cdot\id_F$. Hence after multiplication with an appropriate scalar $\varphi_{\iota}$ satisfies the descent condition and $F$ descends, that is $F\cong \overline{E}$ for some coherent $\mathcal{O}_W$-module $E$. This standard result can be generalized to the noncommutative situation:

\begin{thm}\label{desc}
	Assume $F$ is a simple $\overline{\mathcal{A}}$-module with an isomorphism $F\cong \iota^{*}F$ of $\overline{\mathcal{A}}$-modules, then there is an $\mathcal{A}$-module $E$ and an isomorphism of $\overline{\mathcal{A}}$-modules $F\cong\overline{E}$. 
\end{thm}

The proof of this theorem is the same as the proof of \cite[Theorem 2.6]{fr}. One uses the Brauer-Severi varieties $p: Y \rightarrow W$ and $\overline{p}:\overline{Y}\rightarrow \overline{W}$ associated to $\mathcal{A}$ and $\overline{\mathcal{A}}$. By functoriality of the Brauer-Severi variety and/or the functoriality of the \'{e}tale cyclic Galois cover (as a relative spectrum) we get a morphism $\overline{q}: \overline{Y}\rightarrow Y$.

The idea is to reduce the question about descent of $\overline{\mathcal{A}}$-modules on $\overline{W}$ to a descent argument for classical $\mathcal{O}_{\overline{Y}}$-modules on $\overline{Y}$. This works out well, as the morphism $\overline{q}: \overline{Y}\rightarrow Y$ is also a cyclic \'{e}tale Galois cover. It is induced by the n-torsion line bundle $p^{*}L$. The last fact follows from the injectivity of $p^{*}: \Pic(W) \rightarrow \Pic(Y)$, which in turn follows from the projection formula and $p_{*}\mathcal{O}_Y \cong \mathcal{O}_W$, see \cite[Lemma 1.6]{reede}.

\section{Noncommutative bielliptic surfaces}\label{3}
\begin{defi}
	A smooth projective minimal surface $X$ is called a bielliptic surface if:
	\begin{itemize}
		\item $\kappa(X)=0$ that is $X$ has Kodaira dimension zero,
		\item $q(X)=1$ that is $\mathrm{H}^1(X,\mathcal{O}_X)\cong\mathbb{C}$ and
		\item $p_g(X)=0$ that is $\mathrm{H}^2(X,\mathcal{O}_X)=0$.
	\end{itemize}
\end{defi}
It is well known that each such surface is of the form
\begin{equation*}
	X=(E\times F)/G,
\end{equation*}
where $E$ and $F$ are elliptic curves and $G$ is a finite abelian group. $G$ acts via
\begin{equation*}
	g.(e,f)=(e+g,\Psi(g)(f)),
\end{equation*}
where we understand $G\subset E$ as a finite subgroup and $\Psi: G \rightarrow \Aut(F)$ is an injective group homomorphism. That is $E/G$ is again an elliptic curve and $F/G\cong\mathbb{P}^1$. Furthermore $\omega_X\in \Pic(X)$ is a torsion element of order $n$ with $n\in \left\lbrace 2,3,4,6 \right\rbrace$.

Using this structure theorem Bagnera and de Franchis were able to classify all bielliptic surfaces. In fact, each such surface belongs to one of seven families, which can be found in the following table, see for example \cite[V.5]{hulek} or \cite[List VI.20]{beau4}:

\begin{center}
	\begin{tabular}{cccc}
		\hline \noalign{\vskip 2mm} 
		Type & G & order of $\omega_X$ & $\Br(X)$ \\ [0.5ex] 
		\hline\noalign{\vskip 2mm} 
		1 & $\mathbb{Z}/2\mathbb{Z}$ & 2 & $\mathbb{Z}/2\mathbb{Z}\times \mathbb{Z}/2\mathbb{Z}$ \\ \noalign{\vskip 1mm} 
		2 & $\mathbb{Z}/2\mathbb{Z}\times\mathbb{Z}/2\mathbb{Z}$ & 2 & $\mathbb{Z}/2\mathbb{Z}$ \\\noalign{\vskip 1mm} 
		3 & $\mathbb{Z}/4\mathbb{Z}$ & 4 & $\mathbb{Z}/2\mathbb{Z}$ \\\noalign{\vskip 1mm} 
		4 & $\mathbb{Z}/4\mathbb{Z}\times\mathbb{Z}/2\mathbb{Z}$ & 4 & 0 \\\noalign{\vskip 1mm} 
		5 & $\mathbb{Z}/3\mathbb{Z}$ & 3 & $\mathbb{Z}/3\mathbb{Z}$ \\\noalign{\vskip 1mm} 
		6 & $\mathbb{Z}/3\mathbb{Z}\times\mathbb{Z}/3\mathbb{Z}$ & 3 & 0 \\\noalign{\vskip 1mm} 
		7 & $\mathbb{Z}/6\mathbb{Z}$ & 6 & 0 \\ [1ex] 
		\hline
	\end{tabular}
\end{center}

\begin{rem}
	The Brauer group of a bielliptic surface can be found as follows: first note that according to \cite[Proposition 4]{beau3} there is an isomorphism
	\begin{equation*}
		\Br(X) \cong \mathrm{H}^3(X,\mathbb{Z})_{\mathrm{tor}}.
	\end{equation*}
	Poincar\'{e} duality identifies the last group with $\mathrm{H}_1(X,\mathbb{Z})_{\mathrm{tor}}$. But these groups were, for example, computed by Serrano in \cite[Page 531, Table 3]{ser}.
\end{rem}

The table shows that we only need to work with bielliptic surfaces of type $1,2,3$ and $5$ in the following, since we are interested in nontrivial Azumaya algebras.

Next we want to study the nontrivial elements in the Brauer group of a bielliptic surface. For this we start with a field $K$. A central simple $K$-algebra $A$ is called cyclic if $A$ contains a strictly maximal subfield $L\hookrightarrow A$, such that $L$ is a cyclic Galois extension of $K$. Here $L$ is called strictly maximal if $[L:K]=\deg(A)$. These algebras are special cases of so called crossed products and have the fairly simple description $A=L[t,\sigma]/(t^n-a)$, where $L[t,\sigma]$ is the skew-polynomial ring defined by a generator $\sigma\in \Gal(L/K)$, $a\in K^{\times}$ and $n=\deg(A)$, see for example \cite[Chapter 15]{pierce} or \cite[Part II: Sections 9 and 10]{drax}. 

If $K$ contains a primitive $n$-th root of unity $\zeta$ and $\mathrm{char}(K)\nmid n$, then cyclic algebras can also be described as $n$-symbol algebras, see \cite[Corollary 2.5.5]{gille}. Here for $a,b\in K^{\times}$ the  $n$-symbol algebra $\left\langle a,b \right\rangle_n$ is the $K$-algebra generated by two elements $u,v$ with the relations
\begin{equation*}
	u^n=a,\,\,\,v^n=b\,\,\,\text{and}\,\,\, uv=\zeta vu.
\end{equation*}

The algebra $\left\langle a,b \right\rangle_n$ is central simple and satisfies
\begin{equation*}
	\deg(\left\langle a,b \right\rangle_n)=n\,\,\,\,\text{and}\,\,\,\,\ord(\left\langle a,b \right\rangle_n)=n\,\,\text{in}\,\Br(K).
\end{equation*}

\begin{prop}\label{repr}
	The nontrivial elements in the Brauer group of a bielliptic surface $X$ can be represented by Azumaya algebras $\mathcal{A}$ on $X$ that are generically central simple cyclic division $\mathbb{C}(X)$-algebras such that
	\begin{equation*}
		\deg(\mathcal{A})=\begin{cases}
			2 & \text{if}\,\,\, \ord\left( \left[\mathcal{A} \right] \right)=2\\
			3 & \text{if}\,\,\,\ord\left( \left[\mathcal{A} \right] \right)=3. 
		\end{cases}
	\end{equation*}
	
\end{prop}
\begin{proof}
	Looking at the list of types of bielliptic surface, we see that a nontrivial element $b\in\text{Br}(X)$ has order two or three.
	As $X$ is smooth by \cite[Th\'{e}or\`{e}me 2.4.]{coll} the restriction to the generic point $\eta$ gives an injection
	\begin{equation*}
		r_\eta: \text{Br}(X)\hookrightarrow \text{Br}(\mathbb{C}(X)).
	\end{equation*}
	So the image $r_\eta(b)$ has order two respectively three in $\text{Br}(\mathbb{C}(X))$.
	
	Since every class in $\Br(\mathbb{C}(X))$ contains (up to isomorphism) a unique division algebra, see \cite[12.5. Proposition b]{pierce}, we may assume that $r_{\eta}(b)$ is represented by a central simple division algebra $A$. By a result of Artin and Tate, the division algebra $A$ has index two respectively three as $\mathbb{C}(X)$ has transcendence degree two over $\mathbb{C}$, see \cite[Appendix]{artin}. It remains to note that the index of $A$ is nothing but its degree, since $A$ is a division algebra.
	
	Due to K\"{o}the's theorem, see \cite[Page 64]{drax}, a division algebra over a field $K$ contains a maximal subfield $L$ such that $L/K$ is separable. But maximal subfields in a division algebra are strictly maximal by \cite[13.1 Corollary b]{pierce}. Thus if $\deg(A)=2$ then $A$ contains a strictly maximal subfield $L$ with $[L:\mathbb{C}(X)]=2$, hence $L/\mathbb{C}(X)$ is cyclic Galois and so $A$ is cyclic. The fact that a division algebra of degree three is cyclic is a classical result due to Wedderburn, see \cite[15.6.]{pierce}.
	
	As the class $[A]=r_{\eta}(b)$ comes from $\text{Br}(X)$ it is unramified at every point of codimension one in $X$, and thus by \cite[Th\'{e}or\`{e}me 2.5.]{coll} there is an Azumaya algebra $\mathcal{A}$ on $X$ with $\mathcal{A}\otimes\mathbb{C}(X)=A$ such that $\left[ \mathcal{A}\right]=b$.
\end{proof}

Since the canonical bundle $\omega_X\in\Pic(X)$ is $n$-torsion, it induces a cyclic \'{e}tale Galois cover $\pi: \overline{X}\rightarrow X$ of degree $n$, the so called canonical cover. This cover satisfies the property $\overline{\omega_X}\cong\omega_{\overline{X}}=\mathcal{O}_{\overline{X}}$. It is known that $\overline{X}$ is an abelian surface. More exactly, if $X=(E\times F)/G$ then we see by \cite[2.2.]{tirab}:
\begin{equation*}
	\overline{X}=\begin{cases} E\times F & \text{if $X$ is of type}\,\,\,1,3,5\\
		(E\times F)/H & \text{if $X$ is of type}\,\,\,2\,\,\text{with}\,\, H\cong \mathbb{Z}/2\mathbb{Z}.
	\end{cases}
\end{equation*}

The canonical cover induces a morphism $\pi^{*}: \Br(X)\rightarrow \Br(\overline{X})$, the so called Brauer map. A natural question is, if the Brauer map is injective. Bergstr{\"o}m, Ferrari, Tirabassi and Vodrup give a complete answer to this question in \cite{tirab}. It turns out that the answer is quite complicated and subtle in some cases and the results are not easily stated. Here we only record one example of these results, because it resembles most the case of Enriques surfaces found by Beauville in \cite{beau}, see \cite[Theorem 5.3.]{tirab}:

\begin{thm}
	Let $X$ be a bielliptic surface with $X=(E\times F)/G$. If the elliptic curves $E$ and $F$ are not isogenous, then the morphism $\pi^{*}: \Br(X) \rightarrow \Br(\overline{X})$ is injective.
\end{thm}

Since the property that two elliptic curves $E$ and $F$ are not isogenous is very general in the moduli of these curves, a very general bielliptic surface (in some "moduli" sense) has injective Brauer map. Thus if $X$ is a bielliptic surface with injective Brauer map then the pullback $\overline{\mathcal{A}}$ on $\overline{X}$ of an Azumaya algebra $\mathcal{A}$ constructed in Proposition \ref{repr} represents a nontrivial class in $\Br(\overline{X})$.

\section{Noncommutative Picard schemes and deformations}\label{4}
In this section we first start more generally with a smooth projective $d$-dimensional variety $W$ and an Azumaya algebra $\mathcal{A}$ on $W$. We can think of the pair $(W,\mathcal{A})$ as a noncommutative version of $W$. We want to study moduli schemes of sheaves on such noncommutative pairs.

\begin{defi}
	A sheaf $E$ on $W$ is called a generically simple torsion free $\mathcal{A}$-module, if $E$ is a left $\mathcal{A}$-module such that 
	\begin{itemize}
		\item $E$ is coherent and torsion free as a $\mathcal{O}_W$-module
		\item the stalk $E_{\eta}$ over the generic point $\eta\in W$ is a simple module over $\mathcal{A}_{\eta}$.
	\end{itemize}
	If $\mathcal{A}_{\eta}$ is even a central simple division algebra over $\mathbb{C}(W)$ then such a module is also called a torsion free $\mathcal{A}$-module of rank one. 
\end{defi}

\begin{rem}
	An $\mathcal{A}$-module is locally projective if and only if it is locally free as an $\mathcal{O}_W$-module. If $\mathcal{A}_{\eta}$ is a central simple division algebra then locally projective $\mathcal{A}$-modules of rank one can be thought of as line bundles on the noncommutative variety $(W,\mathcal{A})$. Furthermore a generically simple torsion free $\mathcal{A}$-module is simple, see the argument after Remark 1.1. in \cite{hoff}.
\end{rem}

By fixing the Hilbert polynomial $P$ of such sheaves (with respect to a chosen ample line bundle), Hoffmann and Stuhler showed that these modules are classified by a moduli scheme, see \cite[Theorem 2.4. iii), iv)]{hoff}:

\begin{thm}
	There is a projective moduli scheme $\M_{\mathcal{A}/W,P}$ classifying generically simple  torsion free $\mathcal{A}$-modules with Hilbert polynomial $P$ on $W$.
\end{thm} 

According to \cite[Page 379]{hoff}, the moduli scheme of all generically simple $\mathcal{A}$-modules is given by
\begin{equation*}
	\M_{\mathcal{A}/W}:= \coprod\limits_{P} \M_{\mathcal{A}/W,P}=\coprod\limits_{c_1,\ldots,c_d} \M_{\mathcal{A}/W,c_1,\ldots,c_d}.
\end{equation*}
By the remark above $\M_{\mathcal{A}/W}$ can be understood as the Picard scheme $\Pic(\mathcal{A})$ of the noncommutative variety $(W,\mathcal{A})$ in case the generic stalk is central simple division algebra. 

We also note the following useful facts: 

\begin{rem}\label{eflat}
	For a torsion free $\mathcal{A}$-module $E$ of rank one on $X$, the $\mathcal{A}$-modules $E^{**}$ and $E\otimes L$ for $L\in \Pic(X)$ are also torsion free of rank one. In addition $\overline{E}$ is a torsion free $\overline{\mathcal{A}}$-module of rank one on $\overline{X}$ since $\pi$ is flat.
\end{rem}

We want to study these moduli schemes for a noncommutative bielliptic surfaces $(X,\mathcal{A})$ with injective Brauer map. 

Since the nontrivial elements in $\Br(X)$ can be represented by Azumaya algebras which are generically central simple division algebras, we can work with torsion free $\mathcal{A}$-modules of rank one in the following. Note that the $\mathcal{O}_X$-rank of a torsion free $\mathcal{A}$-module of rank one $E$ is
\begin{equation*}
	\rk_{\mathcal{O}_X}(E)=\begin{cases}
		4 & \text{if}\,\,\ord\left(\left[\mathcal{A} \right]  \right)=2\\
		9 & \text{if}\,\,\ord\left(\left[\mathcal{A} \right]  \right)=3.
	\end{cases}
\end{equation*}

We can now state the main result of this section, whose proof is literally the same as for \cite[Theorem 4.10]{fr}.

\begin{thm}\label{thm2}
	Let $(X,\mathcal{A})$ be noncommutative bielliptic surface with injective Brauer map.
	\begin{enumerate}[i)]
		\item The moduli scheme $\M_{\mathcal{A}/X}$ of torsion free $\mathcal{A}$-modules of rank one is smooth.
		\item Every torsion free $\mathcal{A}$-module of rank one can be deformed into a locally projective $\mathcal{A}$-module, that is the locus $\M_{\mathcal{A}/X}^{lp}$ of locally projective $\mathcal{A}$-modules is dense in $\M_{\mathcal{A}/X}$.
		\item For fixed Chern classes $c_1$ and $c_2$ we have 
		\begin{equation*}
			\dim \M_{\mathcal{A}/X,c_1,c_2}=\begin{cases}
				\,\frac{1}{4}\left(8c_2-3c_1^2 \right) -c_2(\mathcal{A})+1 & \text{if}\,\,\ord\left(\left[\mathcal{A} \right]  \right)=2\\
				\noalign{\vskip9pt}
				\,\frac{1}{9}\left(18c_2-8c_1^2 \right) -c_2(\mathcal{A})+1 & \text{if}\,\,\ord\left(\left[\mathcal{A} \right]  \right)=3.
			\end{cases}
		\end{equation*}
	\end{enumerate}
\end{thm}

\begin{rem}
	Part (2) of Theorem \ref{thm2} is a new phenomenon in the noncommutative case. As noted in \cite[Remark 1.6]{hoff}, in the classical case $\mathcal{A}=\mathcal{O}_X$ just torsion free and locally projective generically simple $\mathcal{A}$-modules lie in different connected components of the moduli scheme. The reason is that locally projective $\mathcal{A}$-modules do not satisfy the valuative criterion for properness if $\mathcal{A}$ is nontrivial. That is the reason why one has to allow for just torsion free $\mathcal{A}$-modules to get a proper noncommutative Picard scheme $\Pic(\mathcal{A})$.
\end{rem}

\section{Lagrangian subschemes}
Let $\iota: \overline{X}\rightarrow \overline{X}$ still be a generator of $\Aut(\overline{X}/X)$. This map induces an automorphism
\begin{equation*}
	\iota^{*}: \M_{\overline{\mathcal{A}}/\overline{X},\overline{c_1},\overline{c_2}} \rightarrow \M_{\overline{\mathcal{A}}/\overline{X},\overline{c_1},\overline{c_2}},\,\,\, \left[F\right]\mapsto \left[\iota^{*}F\right]. 
\end{equation*}
Moreover, using Remark \ref{eflat}, the projection $\pi:\overline{X}\rightarrow X$ induces a morphism
\begin{equation*}
	\pi^{*}: \M_{\mathcal{A}/X,c_1,c_2} \rightarrow \M_{\overline{\mathcal{A}}/\overline{X},\overline{c_1},\overline{c_2}},\,\,\,\, \left[E\right] \mapsto \left[\overline{E}\right],  
\end{equation*}
where $\M_{\overline{\mathcal{A}}/\overline{X},\overline{c_1},\overline{c_2}}$ is the corresponding moduli scheme on the associated canonical cover. By \cite[Theorem 3.6.]{hoff}, the latter moduli space is smooth and posseses a symplectic structure.

Our goal in this section is to understand the morphisms $\iota^{*}$ and $\pi^{*}$ as well as their connection to the symplectic structure.

\begin{thm}
	The image of $\pi^{*}$ coincides with the fixed locus of $\iota^{*}$, that is we have $\Ima(\pi^{*})=\Fix(\iota^{*})$. The latter space is a smooth projective subscheme in $\M_{\overline{\mathcal{A}}/\overline{X},\overline{c_1},\overline{c_2}}$. Furthermore the restriction of the symplectic form $\sigma$ on the tangent bundle of $\M_{\overline{\mathcal{A}}/\overline{X},\overline{c_1},\overline{c_2}}$ to $\Ima(\pi^{*})$ vanishes identically.
\end{thm}

\begin{proof}
	We certainly have $\Ima(\pi^{*})\subset \Fix(\iota^{*})$.
	By Theorem \ref{desc} we also have the inclusion $\Fix(\iota^{*})\subset \Ima(\pi^{*})$. So $\Ima(\pi^{*})=\Fix(\iota^{*})$. The subscheme $\Fix(\iota^{*})$ is projective and smooth by \cite[3.1,3.4]{edix}.
	
	As $\overline{E}$ is simple by Remark \ref{eflat}, using Proposition \ref{serre} and Corollary \ref{homvan} we compute:
	\begin{equation*}
		\Ext^2_{\mathcal{A}}(E,E)\cong \left(\Hom_{\mathcal{A}}(E,E\otimes\omega_X) \right)^{\vee} =0\,\,\,\,\,\text{for all $[E]\in \M_{\mathcal{A}/X,c_1,c_2}$}.
	\end{equation*}
	Now the vanishing of the symplectic form follows similar to \cite[Proof of (3), p.92]{kim} from the following commutative diagram:
	\begin{equation*}
		\begin{CD}
			\Ext^1_{\mathcal{A}}(E,E) \times @. \Ext^1_{\mathcal{A}}(E,E) @>>> \Ext^2_{\mathcal{A}}(E,E)\\
			@V\pi^{*}VV @VV\pi^{*}V @VV\pi^{*}V\\
			\Ext^1_{\overline{\mathcal{A}}}(\overline{E},\overline{E}) \times @. \Ext^1_{\overline{\mathcal{A}}}(\overline{E},\overline{E}) @>>> \Ext^2_{\overline{\mathcal{A}}}(\overline{E},\overline{E})
		\end{CD}
	\end{equation*}
	using Mukai's description of the symplectic form on the tangent bundle of $\M_{\overline{\mathcal{A}}/\overline{X},\overline{c_1},\overline{c_2}}$.
\end{proof}

\begin{rem}
	The vanishing of the symplectic form on $\Ima(\pi^{*})$ can also be seen by noting that $\iota^{*}$ is an antisymplectic automorphism of $\M_{\overline{\mathcal{A}}/\overline{X},\overline{c_1},\overline{c_2}}$. More exactly we have $\iota^{*}\sigma=\zeta_n\sigma$ for a nontrivial $n$-th root of unity $\zeta_n$. This follows as in the proof of \cite[Lemma 4.7.]{fr}. One just has to note that $\iota^{*}$ acts as multiplication by $\zeta_n$ on $H^0(\overline{X},\omega_{\overline{X}})$, since $H^0(X,\omega_X)=0$.
\end{rem}

The following two results also have analogues in the commutative case, see \cite[Proposition 9.1]{nuer}.

\begin{thm}\label{moduli}
	Let $(X,\mathcal{A})$ be noncommutative bielliptic surface with injective Brauer map. The pullback map 
	\begin{equation*}
		\pi^{*}: \M_{\mathcal{A}/X,c_1,c_2} \rightarrow \M_{\overline{\mathcal{A}}/\overline{X},\overline{c_1},\overline{c_2}}
	\end{equation*}
	realizes $\M_{\mathcal{A}/X,c_1,c_2}$ as a finite \'{e}tale cover of the smooth subscheme $\Fix(\iota^{*})\subset \M_{\overline{\mathcal{A}}/\overline{X},\overline{c_1},\overline{c_2}}$.
\end{thm}

\begin{proof}
	The previous theorem shows that $\pi^{*}$ factors through $\Fix(\iota^{*})$ giving rise to a surjective morphism
	\begin{equation*}
		\varphi: \M_{\mathcal{A}/X,c_1,c_2} \rightarrow \Fix(\iota^{*}).
	\end{equation*}
	Since $\pi: \overline{X}\rightarrow X$ is the canonical cover one has isomorphisms
	\begin{equation*}
		\overline{E\otimes\omega_X^j}\cong \overline{E} \,\,\,\text{for any $0\leqslant j \leqslant n-1$}.
	\end{equation*}
	By Corollary \ref{homvan} the $E\otimes\omega_X^j$ are pairwise non-isomorphic since $\overline{E}$ is simple.
	
	Now assume $\varphi(\left[E\right] )=\varphi(\left[F\right] )$ that is $\overline{E}\cong\overline{F}$ and $\Hom_{\overline{\mathcal{A}}}(\overline{E},\overline{F})\cong \mathbb{C}$. Then Lemma \ref{hom} says
	\begin{equation*}
		\mathbb{C}\cong\Hom_{\overline{\mathcal{A}}}(\overline{E},\overline{F}) \cong \bigoplus\limits_{i=0}^{n-1} \Hom_{\mathcal{A}}(E,F\otimes\omega_X^i) 
	\end{equation*}
	and so by \cite[Lemma 4.3]{fr} we have 
	\begin{equation*}
		E\cong F\otimes\omega_X^j
	\end{equation*}
	for exactly one $j$ with $0\leqslant j\leqslant n-1$. So $\varphi$ is an unramified morphism of degree $n$. Moreover the computations also show that $\varphi$ is flat by \cite[Lemma, p.675]{schaps}, hence $\varphi$ is \'{e}tale.
\end{proof}

As the symplectic form vanishes on $\Fix(\iota^{*})$ one may ask if $\Fix(\iota^{*})$ is a Lagrangian subscheme in $\M_{\overline{\mathcal{A}}/\overline{X},\overline{c_1},\overline{c_2}}$. This question can be answered by a simple dimension computation.

\begin{lem}
	Let $(X,\mathcal{A})$ be noncommutative bielliptic surface with injective Brauer map.
	The subscheme $\Fix(\iota^{*})$ of $\M_{\overline{\mathcal{A}}/\overline{X},\overline{c_1},\overline{c_2}}$ is Lagrangian if and only if the bielliptic surface $X$ is of type $1$ or $2$ or $\dim(\M_{\mathcal{A}/X,c_1,c_2})=1$.
\end{lem}
\begin{proof}
	By Theorem \ref{moduli} we need to check in which cases we have 
	\begin{equation*}
		\dim\left( \M_{\overline{\mathcal{A}}/\overline{X},\overline{c_1},\overline{c_2}}\right) =2\dim\left( \M_{\mathcal{A}/X,c_1,c_2}\right).
	\end{equation*}
	It is therefore enough to check in which cases we have
	\begin{equation}\label{lagra}
		\ext^1_{\overline{\mathcal{A}}}(\overline{E},\overline{E})=2\ext^1_{\mathcal{A}}(E,E)
	\end{equation}
	for a locally projective $\mathcal{A}$-module $E$ of rank one by part (2) of Theorem \ref{thm2}. 
	
	Since the canonical cover is finite \'{e}tale of degree $n$ we find using \cite[Lemma 1.3]{fr}:
	\begin{equation*}
		\chi(\mathcal{H}om_{\overline{\mathcal{A}}}(\overline{E},\overline{E}))=\chi(\pi^{*}\mathcal{H}om_{\mathcal{A}}(E,E))=n\chi(\mathcal{H}om_{\mathcal{A}}(E,E)).
	\end{equation*}
	Thus we have
	\begin{equation*}
		2-\ext^1_{\overline{\mathcal{A}}}(\overline{E},\overline{E})=n\left( 1-\ext^1_{\mathcal{A}}(E,E)\right). 
	\end{equation*}
	Inserting equation \eqref{lagra} into the last equation and simplifying gives: 
	\begin{equation*}
		(n-2)\left(1-\ext^1_{\mathcal{A}}(E,E) \right)=0.
	\end{equation*}
	We conclude: $\Fix(\iota^{*})$ is a Lagrangian subscheme if and only if the canonical cover of $X$ has degree two, that is $X$ is of type 1 or 2, or in case $\dim(\M_{\mathcal{A}/X,c_1,c_2})=1$.
\end{proof}

\end{document}